\documentclass[10pt]{amsart}

\usepackage{amsmath, amscd, amssymb}
\usepackage[frame,cmtip,arrow,matrix,line,graph,curve]{xy}
\usepackage{graphpap, color}
\usepackage[mathscr]{eucal}
\usepackage{cancel}
\usepackage{verbatim}

\usepackage{stmaryrd}

\usepackage{tikz,pgf}
\usetikzlibrary{calc,tikzmark,shapes,decorations.pathmorphing,decorations.pathreplacing,decorations.text,decorations,shapes.symbols,arrows,positioning}

\numberwithin{equation}{section}

\newcommand{\CC}{\mathbb{C}}

\newcommand{\PP}{\mathbb{P}}

\newcommand{\ZZ}{\mathbb{Z}}

\newcommand{\cal}{\mathcal}

\def\cO{{\cal O}}

\def\cQ{{\cal Q}}
\def\cR{{\cal R}}

\def\cU{{\cal U}}

\def\and{\quad{\rm and}\quad}

 \DeclareMathOperator{\Ext}{Ext}
  \DeclareMathOperator{\Hom}{Hom}

\DeclareMathOperator{\rk}{rk}

\def\to{\rightarrow}

\def\Hom{\operatorname{Hom}}

\def\Sym{\operatorname{Sym}}

\def\Coker{\operatorname{Coker}}
\def\dim{\operatorname{dim}}
\def\Gr{\operatorname{Gr}}

\def\GL{\operatorname{GL}}

\newtheorem{prop}{Proposition}[section]
\newtheorem{theo}[prop]{Theorem}
\newtheorem{lemm}[prop]{Lemma}
\newtheorem{coro}[prop]{Corollary}

\theoremstyle{definition}
\newtheorem{defi}[prop]{Definition}

\def\beq{\begin{equation}}
\def\eeq{\end{equation}}

\def\PP{\mathbb{P}}
\def\CC{\mathbb{C}}

\def\cO{\mathcal{O}}

\title[Moduli spaces of Ulrich bundles on the Fano 3-fold $V_5$]{Moduli spaces of Ulrich bundles \\ on the Fano 3-fold $V_5$}

\author{Kyoung-Seog Lee} 
\address{Center for Geometry and Physics, Institute for Basic Science (IBS), Pohang 37673, Republic of Korea 
\vskip 0.2em 
Current Address: 
Institute of the Mathematical Sciences of the Americas, University of Miami, 1365 Memorial Drive, Ungar 515, Coral Gables, FL 33146, USA}
\email{kyoungseog02@gmail.com}

\author{Kyeong-Dong Park}
\address{Center for Geometry and Physics, Institute for Basic Science (IBS), Pohang 37673, Republic of Korea
\vskip 0.2em 
Current Address: 
School of Mathematics, Korea Institute for Advanced Study (KIAS), Dongdaemun-gu, Seoul 02455, Republic of Korea}
\email{kdpark@kias.re.kr}

\thanks{}



\begin{document}

\begin{abstract}
We study moduli spaces of Ulrich bundles of rank $r \geq 2$ on the Fano 3-fold $V_5$ of Picard number 1, degree 5 and index 2. 
We prove that the moduli space of stable Ulrich bundles of rank $r$ on $V_5$ can be identified with a smooth $(r^2+1)$-dimensional open subset of the moduli space of stable quiver representations with dimension vector $(r, r)$ of the Kronecker quiver with 2 vertices and 3 arrows.
\end{abstract}

\maketitle

\section{Introduction}

Ulrich bundles form interesting classes of vector bundles on higher dimensional algebraic varieties. 
Recently, existence and moduli of Ulrich bundles have drawn a lot of attention. 
It turns out that a large class of algebraic varieties admit Ulrich bundles (see \cite{Beauville} for more details) and 
it seems that the moduli space of Ulrich bundles of a given variety reflects many interesting features of the variety.

Kuznetsov studied instanton bundles on some Fano 3-folds via their bounded derived categories of coherent sheaves in \cite{Kuznetsov}. 
Using derived categories, Lahoz, Macr\`{i} and Stellari studied moduli spaces of Ulrich bundles on cubic 3-folds and 4-folds in \cite{LMS1, LMS2}, 
and Cho, Kim and the first named author studied Ulrich bundles on intersections of two 4-dimensional quadrics 
in \cite{CKL}. 

In this paper, we study moduli spaces of Ulrich bundles on $V_5$ which is the unique smooth Fano 3-fold of Picard number 1, degree 5 and index 2. 
The 
variety $V_5$ is a very interesting Fano 3-fold which enjoy many beautiful geometric and topological properties. 
Vector bundles on $V_5$ was studied by Arrondo and Costa in \cite{AC}, Faenzi in \cite{Faenzi}, and Kuznetsov in \cite{Kuznetsov}. 
The bounded derived category $\mbox{D}(V_5)$ of coherent sheaves on $V_5$ was studied by Orlov in \cite{Orlov} where he proved that there are full exceptional collections in $\mbox{D}(V_5)$. 
From one of the exceptional collections and Bondal's work (cf. \cite{Bondal}), 
we see that the derived category of the finitely generated modules of the path algebra associated to the Kronecker quiver $\Gamma$ with 2 vertices and 3 arrows is embedded in $\mbox{D}(V_5)$. 
Using this semiorthogonal decomposition of the derived category of $V_5$, we obtain the following result  
(see Theorem \ref{moduli space of Ulrich} for more precise statement).

\begin{theo}
For any $r \geq 2,$ there is a nontrivial $(r^2+1)$-dimensional family of stable Ulrich bundles of rank $r$ on $V_5.$ 
Moreover, the moduli space of stable Ulrich bundles of rank $r$ on $V_5$ is isomorphic to a smooth open subset of the moduli space of stable quiver representations of the quiver $\Gamma$ with dimension vector $(r, r)$. 
In particular, $V_5$ is of Ulrich-wild representation type.
\end{theo}

\bigskip

\noindent
\textbf{Conventions}.
We will work over $\CC.$ For a variety $X,$ we will use $\mbox{D}(X)$ to denote the bounded derived category of coherent sheaves on $X.$ 
For a quiver $\Gamma,$ we will consider only finite-dimensional representations of $\Gamma.$

\bigskip

\noindent
\textbf{Acknowledgments}. 
This work was motivated by talks given by Alexander Kuznetsov and Paolo Stellari. 
The first named author thanks to Alexander Kuznetsov and Paolo Stellari for kind explanations and motivating discussions. 
This paper also use many ideas and results from the earlier works (especially \cite{CKL, Faenzi, Kuznetsov}) mentioned above. 
The first named author thanks Mudumbai Seshachalu Narasimhan for helpful comments and encouragements, Yonghwa Cho and Yeongrak Kim for useful discussions during the preparation of \cite{CKL} and Ludmil Katzarkov and Simons Foundation for partially supporting this work via Simons Investigator Award-HMS 503875.
We thank Jaeyoo Choy, Rosa Maria Mir\'{o}-Roig and Se-jin Oh for helpful discussions and encouragements.
Finally, we would like to express our thanks to the referees for their comments and pointing out several mistakes that helped us to improve the paper. 

This work was partially supported by the Institute for Basic Science (IBS-R003-Y1) in Korea. 
The second author was partially supported by the Institute for Basic Science (IBS-R003-D1) and by the National Research Foundation of Korea (NRF) grant funded by the Korea government (MSIT) (NRF-2019R1A2C3010487). 

\bigskip

\section{Preliminaries}

In this section we briefly recall basic definitions and facts about Ulrich bundles, quiver representations and the Fano 3-fold $V_5.$

\subsection{Ulrich bundles}

Let us recall the definition of Ulrich bundles on smooth projective varieties.

\begin{defi}
Let $X\subset \mathbb P^N$ be a smooth projective variety of dimension $d.$ 
An \emph{Ulrich bundle} $E$ is a vector bundle on $X$ satisfying 
$$ H^i(X,E(-j))=0 $$
for all $i=0, 1, \cdots, d$ and $j=1, \cdots, d.$ 
Here, we denote the twisted bundle $E \otimes \mathcal O_X(-j)$ by $E(-j)$.
\end{defi}

Ulrich bundles enjoy many nice properties, and existence and moduli spaces of Ulrich bundles on a given variety tell us many things about the variety. 
See \cite{Beauville} and references therein for more details about Ulrich bundles.


For use in Section 3, let us recall the following result in \cite{CHGS} among properties of Ulrich bundles. 
We refer for instance to \cite{HL} for details about semistable sheaves.  

\begin{theo}[Theorem 2.9 of \cite{CHGS}]
\label{semistable bundle}
Let $E$ be an Ulrich bundle on a smooth projective variety $X.$ 
Then $E$ is a semistable bundle.
\end{theo}

\subsection{Quiver representations}

Let us recall some basic definitions and properties of quiver representations as follows. 
See \cite{CB} for more discussions about quiver representations.

\begin{defi}
A \emph{quiver} $\Gamma$ is a finite directed graph, i.e., $\Gamma=(\Gamma_0,\Gamma_1)$ consists of a finite set $\Gamma_0$ of vertices and a finite set $\Gamma_1$ of arrows (possibly multiple arrows and loops).
\end{defi}

For example, the \emph{Kronecker quiver} with $r$ arrows is a quiver having two vertices and $r$ arrows pointing in the same direction. 


\begin{defi}
Let $\Gamma$ be a quiver. 
A \emph{representation $V_{\Gamma}$ of the quiver $\Gamma$} is a collection of vector spaces $V_i$ for each vertex $i \in \Gamma_0$ 
and a collection of linear maps $f^V_{ij} \colon V_i \to V_j$ for each arrow $\rho_{ij} \in \Gamma_1.$
A \emph{subrepresentation} $W_{\Gamma}$ of $V_{\Gamma}$ is a quiver representation with injective linear maps $\iota_i \colon W_i \to V_i$ 
such that the following diagram commutes for each $\rho_{ij} \in \Gamma_1.$
\begin{displaymath}
\xymatrix{ 
W_i \ar[r]^{f^W_{ij}} \ar[d]_{\iota_i} & W_j \ar[d]^{\iota_j} \\
V_i \ar[r]^{f^V_{ij}} & V_j}
\end{displaymath}
\end{defi}

When we have a quiver $\Gamma$ and a sequence of integers $\underline{d}=(d_i)_{i \in \Gamma_0},$ 
it is an interesting task to classify $\Gamma$-representations $V_{\Gamma}$ 
which have a given dimension vector $\underline{d}^{V_{\Gamma}}=(d_i)_{i \in \Gamma_0},$ i.e., each vector space $V_i$ has dimension $d_i.$ 
Two standard ways will be constructing the moduli space of such representations using stacks or introducing certain stability conditions on the $\Gamma$-representations. 
Indeed, there are well-known stability conditions studied by King as follows (cf. \cite{King, Reineke}).

\begin{defi}
Let $\Gamma$ be a quiver and $\theta \colon \ZZ^{|\Gamma_0|} \to \ZZ$ be a $\ZZ$-linear map. 
Then a $\Gamma$-representation $V_{\Gamma}=(V_i,f_{ij})$ is \emph{$\theta$-(semi)stable} 
if $\theta(\underline{d}^{W_{\Gamma}}) < \theta(\underline{d}^{V_{\Gamma}})$ (respectively, $\theta(\underline{d}^{W_{\Gamma}}) \leq \theta(\underline{d}^{V_{\Gamma}})$) 
for every nonzero proper subrepresentation $W_{\Gamma}$ of $V_{\Gamma}.$
\end{defi}

Using the above $\theta$-stability condition, King constructed moduli spaces of $\theta$-semistable $\Gamma$-representations via Geometric Invariant Theory in \cite{King}. 
From now on, let us use $\mathcal M_{\underline{d}}^{{\theta}\text{-ss}}(\Gamma)$ (resp. $\mathcal M_{\underline{d}}^{{\theta}\text{-s}}(\Gamma)$) 
to denote the moduli space of $\theta$-semistable (resp. $\theta$-stable) $\Gamma$-representations with dimension vector $\underline{d}.$ 
Let us recall King's construction of $\mathcal M_{\underline{d}}^{{\theta}\text{-ss}}(\Gamma)$ and $\mathcal M_{\underline{d}}^{{\theta}\text{-s}}(\Gamma).$ 
(See \cite{King, Reineke} and references therein for more details.) 
Let $\cR_{\underline{d}}(\Gamma):=\bigoplus_{\rho_{ij} \in \Gamma_1} \Hom(V_i,V_j)$ be the vector space of $\Gamma$-representations with dimension vector $\underline{d}.$ 
For each vertex $i \in \Gamma_0,$ we have a national $\GL(d_i)$-action on $V_i$ and 
$G_{\underline{d}}:=\prod_{i \in \Gamma_0} \GL(d_i)$ acts on $\cR_{\underline{d}} (\Gamma)$ via the base change, i.e., $(g_i) \cdot (f_{ij}):=(g_j f_{ij} g_i^{-1}).$ 
Because the diagonal $\CC^*$-action on each $V_i$ acts trivially on $\cR_{\underline{d}} (\Gamma)$, 
we have a $G_0$-action on $\cR_{\underline{d}}(\Gamma)$ where $G_0=G_{\underline{d}} / \CC^*.$ 
Consider a trivial line bundle $L_0$ on $\cR_{\underline{d}}(\Gamma).$ 
For each $\theta \colon \ZZ^{|\Gamma_0|} \to \ZZ$ such that $\theta(\underline{d})=0,$ there is a corresponding $G_0$-linearization $\chi$ on $L_0.$ 
Let $\cR_{\underline{d}}^{{\theta}\text{-ss}}(\Gamma)$ (resp. $\cR_{\underline{d}}^{{\theta}\text{-s}}(\Gamma)$) be the set of semistable (resp. stable) points in $\cR_{\underline{d}}(\Gamma).$ 
King constructed the coarse moduli space $\mathcal M_{\underline{d}}^{{\theta}\text{-ss}}(\Gamma)$ as the GIT quotient $\cR_{\underline{d}}(\Gamma) \sslash_{\chi} G_0.$ 
Therefore, $\mathcal M_{\underline{d}}^{{\theta}\text{-ss}}(\Gamma)$ is the good quotient $\cR^{\theta\text{-ss}}_{\underline{d}}(\Gamma) \sslash_{\chi} G_0$ and 
$\mathcal M_{\underline{d}}^{{\theta}\text{-s}}(\Gamma)$ is the geometric quotient $\cR^{\theta\text{-s}}_{\underline{d}}(\Gamma) \sslash_{\chi} G_0.$

\begin{theo}[\cite{King}]
Fix $\Gamma$, $\underline{d}$ and $\theta$ as above. 
Then there exists a coarse moduli space $\mathcal M_{\underline{d}}^{{\theta}\text{-ss}}(\Gamma)$ (resp. $\mathcal M_{\underline{d}}^{{\theta}\text{-s}}(\Gamma)$) parametrizing S-equivalent classes of $\theta$-semistable (resp. $\theta$-stable) $\Gamma$-representations with dimension vector $\underline{d}$, 
which is an irreducible normal projective (resp. quasi-projective) variety.
\end{theo}

When we have a quiver $\Gamma$, we also have an associated path algebra $\CC \Gamma$ and it is well-known that there is a natural equivalence between the category of quiver representations and the category of $\CC \Gamma$-modules. 
The following result is also well-known 
(see \cite{CB} for more details).

\begin{prop}[\cite{CB}] 
\label{path algebra and quiver representation}
Let $\Gamma$ be a quiver. 
\begin{enumerate}
\item 
For the path algebra $\CC \Gamma$ associated to a quiver $\Gamma$, let $\CC \Gamma${\rm-mod} be the category of $\CC \Gamma$-modules. 
Then the homological dimension of $\CC \Gamma${\rm-mod} is at most 1. 
In particular, $\CC \Gamma${\rm-mod} is a hereditary category.

\item 
Let $A_{\Gamma}$ (resp. $B_{\Gamma}$) be a finite-dimensional quiver representation of $\Gamma$, and let $\underline{a}$ (resp. $\underline{b}$) be its dimension vector. 
Then we have 
\begin{align*}
\chi(A_{\Gamma}, B_{\Gamma}) & = \dim \Hom(A_{\Gamma},B_{\Gamma}) - \dim \Ext^1(A_{\Gamma},B_{\Gamma}) \\
& = \sum_{i \in \Gamma_0} a_i b_i - \sum_{\rho_{ij} \in \Gamma_1} a_{i}b_{j} = \chi(\underline{a},\underline{b}).
\end{align*}
\end{enumerate}
\end{prop}

We can compute extensions between quiver representations using the above result. 
Moreover, using general arguments in homological algebra we have the following corollary.

\begin{coro}
Any element $F \in \mbox{\rm D}(\CC \Gamma\mbox{\rm-mod})$ in the derived category of $\CC \Gamma$-modules is isomorphic to $\bigoplus_i H^i(F)[-i].$
\end{coro}

\subsection{Fano 3-fold $V_5$}

The Fano 3-fold $V_5$ is the unique smooth Fano 3-fold of Picard number 1, degree 5 and index 2. 
It is well-known that $V_5$ is a generic codimension 3 linear section of the Grassmannian $\Gr(2,5)$  of 2-dimensional subspaces in a 5-dimensional complex vector space $V$. 
The cohomology group of $V_5$ over $\ZZ$ is isomorphic to $\ZZ^4$ as an abelian group as follows.
\begin{align*}
 H^*(V_5,\ZZ) & = H^0(V_5,\ZZ) \oplus H^2(V_5,\ZZ) \oplus H^4(V_5,\ZZ) \oplus H^6(V_5,\ZZ) \\
& = \ZZ[V_5] \oplus \ZZ[h] \oplus \ZZ[l] \oplus \ZZ[p], 
\end{align*}
where $h \cdot l = p,$ $h^2=5 l$ and $h^3=5p.$ See \cite{Faenzi, IP, Kuznetsov, Orlov} for more details about the Fano 3-fold $V_5.$
Because $V_5$ is a linear section of codimension 3 in $\Gr(2,5)$, we have the following natural exact sequence of vector bundles on $V_5$:
$$ 0 \to U \to V \otimes \mathcal{O} \to Q \to 0, $$
where $U$ is the restriction of the universal 
subbundle $\cU$ on $\Gr(2,5)$ to $V_5$ and $Q$ is the restriction of the universal 
quotient bundle $\cQ$ on $\Gr(2,5)$ to $V_5.$ 
We can compute the Chern classes of $U,U^*,Q,Q^*$ as follows (cf. Section 3 of \cite{Faenzi}). \\
(1) The Chern classes of $U$ are $\rk(U)=2,$ $c_1(U)=-h$ and $c_2(U)=2l.$ \\
(2) The Chern classes of $U^*$ are $\rk(U^*)=2,$ $c_1(U^*)=h$ and $c_2(U^*)=2l.$ \\
(3) The Chern classes of $Q$ are $\rk(Q)=3,$ $c_1(Q)=h,$ $c_2(Q)=3l$ and $c_3(Q)=p.$ \\
(4) The Chern classes of $Q^*$ are $\rk(Q^*)=3,$ $c_1(Q^*)=-h,$ $c_2(Q^*)=3l$ and $c_3(Q^*)=-p.$ \\

Using the Borel--Weil--Bott theorem in \cite{Bott} and the Koszul resolution 
$$ 0 \to \cO(-3) \to \cO(-2)^{\oplus 3} \to \cO(-1)^{\oplus 3} \to \cO_{\Gr(2, 5)} \to \cO_{V_5} \to 0, $$
we can explicitly compute the cohomology groups $H^i(V_5,U(j))$ and $H^i(V_5,Q^*(j))$ for $j=-2,-1,0,1$. 
We can also prove the following lemma using the Riemann--Roch theorem 
since we know that the tautological bundles $U$ and $Q$ have vanishing intermediate cohomology from \cite{AC, Faenzi}. 
We call a vector bundle with this property \emph{arithmetically Cohen--Macaulay} (ACM).
For instance, since $Q^*(1)$ has Chern classes $c_1 = 2 h, c_2 = 8 l$ and $c_3 = 2 p$, 
the Euler characteristic of $Q^*(1)$ is equal to $10$ by the formula in Subsection 2.3 of \cite{AC}. 
From the fact that $Q^*(1)$ is ACM, this result implies $h^0(V_5, Q^*(1)) = 10$.

\begin{lemm}
\label{computation of cohomology}
The cohomology groups $H^i(V_5,U(j))$ and $H^i(V_5,Q^*(j))$ for $j=-2,-1,0,1$ are as follows.
\begin{align*}
& H^i(V_5,U(1))=\left \{ {\begin{array}{ll} 
\CC^5 & \textrm{if $i=0$,} \\ 
0 & \textrm{otherwise} 
\end{array}}
\right. \\
& H^i(V_5,Q^*(1))=\left \{ {\begin{array}{ll} 
\CC^{10} & \textrm{if $i=0$,} \\ 
0 & \textrm{otherwise} 
\end{array}}
\right. \\
& H^*(V_5,U) = H^*(V_5,U(-1)) = 0  \\
& H^*(V_5,Q^*) = H^*(V_5,Q^*(-1)) = 0 \\
& H^i(V_5,U(-2))=H^i(V_5,Q^*(-2))=\left \{ {\begin{array}{ll} 
\CC^5 & \textrm{if $i=3$,} \\ 
0 & \textrm{otherwise.} 
\end{array}}
\right.
\end{align*}
\end{lemm}

\vskip 1em 

Orlov proved that there are full exceptional collections of length 4 on $V_5$.

\begin{prop}[Theorem 1 of \cite{Orlov}] \label{derived category}
There are following full exceptional collections in the derived category $\mbox{\rm  D}(V_5)$ of coherent sheaves on $V_5$.
$$ \mbox{\rm  D}(V_5) = \langle U, Q^*, \mathcal{O}, \mathcal{O}(1) \rangle 
= \langle \mathcal{O}(-2), \mathcal{O}(-1), (V/U)(-1), U \rangle $$
\end{prop}

\vskip 1em

From Bondal's theorem in \cite{Bondal}, 
we see that the triangulated subcategory $\langle U, Q^* \rangle$ is equivalent to the derived category $\mbox{D}(\CC \Gamma\mbox{-mod})$ of $\CC \Gamma$-modules, 
where $\Gamma$ is the Kronecker quiver with two vertices and three arrows between them. 
From now on $\Gamma$ will denote this specific quiver. See \cite{Kuznetsov, Orlov} for more details. \\

Let us consider the following diagram with two projections $p$ and $q$.

\begin{displaymath}
\xymatrix{ 
 & V_5 \times V_5 \ar[ld]_p \ar[rd]^q & \\
V_5 & & V_5}
\end{displaymath}

Faenzi obtained a resolution of the diagonal over $V_5 \times V_5$.

\begin{prop}[Theorem 4.1 and Corollary 5.4 of \cite{Faenzi}] \label{resolution of diagonal}
There is a resolution of the diagonal over $V_5 \times V_5$ as follows.
$$ 0 \to \cO(-1,-1) \to U \boxtimes \wedge^2 Q^* \to Q^* \boxtimes U \to \cO \to \cO_{\Delta} \to 0 $$
By mutation of the exceptional collection, we have another resolution 
$$ 0 \to U(-1) \boxtimes U \to Q^*(-1) \boxtimes Q^* \to \cO(-1) \boxtimes \Omega^1_{\PP^6}(1)|_{V_5} \to \cO \to \cO_{\Delta} \to 0 .$$
\end{prop}

The above resolution will be very useful to compute projection functors.

\bigskip

\section{Ulrich bundles on $V_5$}

In this section we discuss existence and moduli spaces of rank $r$ Ulrich bundles on $V_5$ for any $r \geq 2.$ 

\subsection{Generalities on Ulrich bundles on $V_5$}

First, it is easy to see that there is no Ulrich line bundle on $V_5.$
Moreover, we can see that $(-1)$-twists of Ulrich bundles lie on the specific semiorthogonal component.

\begin{lemm}
Let $E$ be an Ulrich bundle on $V_5.$ Then $E(-1) \in \langle U, Q^* \rangle.$
\end{lemm}
\begin{proof}
By the definition of Ulrich bundles, we see that 
$$H^*(V_5,E(-1)) = H^*(V_5,E(-2)) = 0.$$ 
Hence we have that $E(-1) \in \langle U, Q^* \rangle$ by Proposition \ref{derived category}.
\end{proof}
Therefore, we can study Ulrich bundles on $V_5$ in terms of $U, Q^*.$

\subsection{Rank 2 and 3 cases}

Beauville proved that seven families of Fano 3-folds of even index 
admit Ulrich bundles of rank 2 and computed their deformations via Serre correspondence.

\begin{prop}[Proposition 6.1 and Proposition 6.4 of \cite{Beauville}]
\label{rank 2 Ulrich}
There are Ulrich bundles of rank 2 on $V_5$, and the moduli space of rank 2 special Ulrich bundles is smooth of dimension 5.
\end{prop}


Arrondo and Costa gave an explicit Ulrich bundle of rank 3 on $V_5$ in \cite{AC}. 
As obtained in Example 4.4 of \cite{AC}, the symmetric square $\Sym^2 U^*$ of the dual bundle $U^*$ is an ACM bundle 
and its Chern classes are $c_1 = 3 h, c_2 = 18 l$ and $c_3 = 8p$. 
Then we get $h^0(V_5, \Sym^2 U^*) = 15$ by the Riemann--Roch theorem, 
which implies that $\Sym^2 U^*$ is an Ulrich bundle since $h^0(V_5, \Sym^2 U^*) = \deg(V_5) \cdot \rk(\Sym^2 U^*)$.
On the other hand, this result can be deduced from the fact that the restriction of an Ulrich bundle to a general hyperplane section is Ulrich 
because the equivariant vector bundle $\Sym^2 \mathcal U^*$ is an Ulrich bundle on the Grassmannian $\Gr(2,5)$ by the Borel--Weil--Bott theorem.


\begin{prop}[Section 4 of \cite{AC}] 
\label{rank 3 Ulrich}
There is an Ulrich bundle of rank 3 on $V_5.$
\end{prop}

Faenzi \cite{Faenzi} also studied ACM bundles on $V_5$ and classified semistable ACM bundles on $V_5$ up to rank 3. 

\subsection{Higher rank cases}

Suppose that there exists an Ulrich bundle of rank $r$ on $V_5.$ Because we know that $E(-1)$ lies on $\langle U, Q^* \rangle$, we can express $E(-1)$ in terms of $U$ and $Q^*.$ Let us compute the image of $E(-1)$ in $\langle U, Q^* \rangle$ in terms of quiver representations as follows.

\begin{prop} \label{E(-1)}
For any $r \geq 2,$ an Ulrich bundle $E$ of rank $r$ on $V_5$ corresponds to the following quiver representation. 
$$ E(-1)=\Coker(U^{\oplus r} \to Q^{* \oplus r}) $$ 
\end{prop}

\begin{proof}
Recall that the tensor product of semistable bundles is semistable (cf. \cite[Theorem 3.1.4]{HL}). 
Since an Ulrich bundle on $V_5$ is semistable by Theorem \ref{semistable bundle}, 
$U \otimes E$ and $Q \otimes E$ are semistable. Also note that $c_1(E(-1))=0$ (cf. \cite{CHGS}) and hence $E(-1)$ is normalized. 
Applying the argument used to prove Theorem 6.3 of \cite{Faenzi}, $E(-1)$ splits into the direct sum of two bundles $E_1$ and $E_2$ defined by 
\[
0 \to H^1(Q(-1) \otimes E(-1)) \otimes U \to H^1(U \otimes E(-1)) \otimes Q^* \oplus H^0(E(-1)) \otimes \mathcal O \to E_1 \to 0, 
\]
\[
0 \to E_2 \to H^2(Q(-1) \otimes E(-1)) \otimes U \to H^2(U \otimes E(-1)) \otimes Q^* \oplus H^0(E(-1)) \otimes \mathcal O \to 0. 
\]
From the exact sequence
$$ 0 \to Q^* \to \cO^{\oplus 5} \to U^* \to 0 $$
and the isomorphism $U^* \cong U(1) $, 
we have the following exact sequence
$$ 0 \to Q^* \otimes E(-2) \to E(-2)^{\oplus 5} \to U \otimes E(-1) \to 0. $$
Then we have an isomorphism $H^2(U \otimes E(-1)) \cong H^3(Q^* \otimes E(-2))$ since $E$ is an Ulrich bundle.
The Serre duality and stability condition imply $H^3(Q^* \otimes E(-2)) \cong H^0(Q \otimes E^*)^* = 0$. 
Hence $E_2$ is isomorphic to a direct sum of copies of $U$ because $H^0(E(-1)) = 0$. 
However, since $U(1)$ is not Ulrich, $E(-1)$ does not contain $U$ as a direct summand. 
This yields that $E_2 = 0$ and there exist integers $a, b$ such that $E(-1)$ is the cokernel of $U^{\oplus a} \to Q^{* \oplus b}$.
Comparing the ranks of bundles in the exact sequence 
$$ 0 \to U^{\oplus a} \to Q^{* \oplus b} \to E(-1) \to 0, $$
we obtain that $3b - 2a = r$. 
The above exact sequence twisted by $\mathcal O_{V_5}(-2)$ induces an isomorphism 
$$ 0 = H^2(E(-3)) \to H^3(U(-2)^{\oplus a}) \to H^3(Q^*(-2)^{\oplus b}) \to H^3(E(-3)) = 0. $$
Since $H^3(U(-2)) = H^3(Q^*(-2)) = \mathbb C^5$ by Lemma \ref{computation of cohomology}, we have $a = b = r$, which complete the proof. 
\end{proof}

A similar strategy to the above proof was used earlier in \cite{CKL, Faenzi, Kuznetsov}. 
Especially, Faenzi gave a similar description of general ACM bundles in \cite{Faenzi}. 

\vskip 1em 

Let us compute the image of the line bundle $\mathcal{O}(2)$ on $V_5$ in the triangular subcategory $\langle U, Q^* \rangle$ of the derived category $\mbox{D}(V_5)$. 

\begin{lemm}
The image of $\cO(2)$ in $\langle U, Q^* \rangle$ is $\Coker(U^{\oplus 5} \to {Q^*}^{\oplus 10})[2].$
\end{lemm}

\begin{proof}
As a quiver representation, to find the image of $\cO(2)$ in $\langle U, Q^* \rangle$ is equivalent to find the image of $\cO(1)$ in $\langle U(-1), Q^*(-1) \rangle.$
From the 
resolution of the diagonal over $V_5 \times V_5$ in Proposition \ref{resolution of diagonal}, 
we have the following exact sequence
\begin{align*}
0 & \to U(-1) \otimes H^*(V_5,U(1)) \to Q^*(-1) \otimes H^*(V_5,Q^*(1)) \\
& \to \cO(-1) \otimes H^*(V_5,\Omega^1_{\PP^6}(2)|_{V_5}) \to \cO \otimes H^*(V_5,\cO(1)) \to \cO(1) \to 0 
\end{align*}
on $V_5.$
We already in Lemma \ref{computation of cohomology} saw that 
\begin{displaymath}
H^i(V_5,U(1))=\left \{ {\begin{array}{ll} 
\CC^5 & \textrm{if $i=0$,} \\ 
0 & \textrm{otherwise,} 
\end{array}}
\right.
\end{displaymath}
and
\begin{displaymath}
H^i(V_5,Q^*(1))=\left \{ {\begin{array}{ll} 
\CC^{10} & \textrm{if $i=0$,} \\ 
0 & \textrm{otherwise.} 
\end{array}}
\right.
\end{displaymath}
Therefore, we see that the image of $\cO(2)$ in $\langle U, Q^* \rangle$ is $\Coker(U^{\oplus 5} \to {Q^*}^{\oplus 10})[2].$
\end{proof}

From now on, let us use $R_{\bullet}$ to denote the quiver representation of $\Gamma$ corresponding to the image of $\mathcal{O}(2)$ in $\langle U, Q^* \rangle.$

\begin{prop}\label{dimension}
Let $r \geq 2$ be an integer. Suppose that there exists a stable Ulrich bundle of rank $r$ on $V_5.$ Then the moduli space of stable Ulrich bundles of rank $r$ is smooth of dimension $r^2+1.$ 
\end{prop}
\begin{proof}
From the presentation of a stable Ulrich bundle $E$ of rank $r$ via quiver representation (Proposition \ref{E(-1)})
$$ E(-1)=\Coker(U^{\oplus r} \to {Q^*}^{\oplus r}),  $$ 
we see that
$$ ext^2(E,E)=ext^3(E,E)=0 $$
and
$$ ext^0(E,E) - ext^1(E,E)= \chi(E,E)=(r^2 + r^2) -3r^2 = -r^2 $$
by Proposition \ref{path algebra and quiver representation}. 
Note that the path algebra associated to the quiver is \emph{hereditary}, i.e., $\Ext^i(A_{\Gamma}, B_{\Gamma})=0$ for any $i \geq 2$ and $\Gamma$-representations $A_{\Gamma}, B_{\Gamma}.$ 
Therefore, the moduli space of stable rank $r$ Ulrich bundles is smooth of dimension $r^2+1.$
\end{proof}

It remains to prove the existence of rank $r$ Ulrich bundles on $V_5$ for any $r \geq 2.$ 
From the Serre correspondence in \cite{Beauville} and explicit construction in \cite{AC}, we know that there are rank 2 and rank 3 Ulrich bundles on $V_5$ (Propositions \ref{rank 2 Ulrich} and \ref{rank 3 Ulrich}). 
Then we can prove it via deformation argument which was also used in \cite{CHGS, CKL}. 
The argument is almost the same as that of \cite[Theorem 5.7]{CHGS}, possibly except the computation of $ext^i.$ 
We write the proof for the convenience of readers.


\begin{prop}\label{existence}
For any $r \geq 2$, there is a stable Ulrich bundle of rank $r$ on $V_5.$
\end{prop}
\begin{proof}
We already saw that there exist Ulrich bundles of rank 2 and rank 3 on $V_5$. 
Because there is no Ulrich line bundle on $V_5,$ rank 2 and rank 3 Ulrich bundles are stable. 
For $r \geq 4,$ we will use induction to prove the existence of rank $r$ Ulrich bundles on $V_5$. 

Let us consider a stable rank 2 Ulrich bundle $E_1$ and a stable rank $r-2$ Ulrich bundle $E_2$ 
which are non-isomorphic to each other and have a non-trivial extension of the following form 
$$ 0 \to E_1 \to E \to E_2 \to 0. $$
Because every Ulrich bundle arises from a quiver representation, 
$ext^i(E_2,E_1)=0$ for $i=2,3$. 
We also have $\Hom(E_2, E_1) = \Hom(E_2(-1), E_1(-1))=0$ since $E_1(-1)$ and $E_2(-1)$ are non-isomorphic stable bundles with the same slope. 
Therefore, we see that $ext^1(E_2,E_1)=-\chi(E_2,E_1)=2(r-2)$ and 
there are non-trivial such extensions. 

We see that $E$ is a simple vector bundle from \cite[Lemma 4.2]{CHGS} and there is an $(r^2+1)$-dimensional smooth modular family of simple vector bundles having the same rank and Chern classes with $E$ from \cite[Proposition 2.10]{CHGS}. 
Being Ulrich bundle is an open condition, so there is an $(r^2+1)$-dimensional family of Ulrich bundles among the modular family. 
Via similar computations as above, 
we see that the locus of strictly semistable Ulrich bundles 
arising as extensions of rank $r_1$ Ulrich bundles by rank $r-r_1$ Ulrich bundles has dimension at most $r_1^2+1 + (r-r_1)^2+1+r_1(r-r_1)=r^2-r_1r+r_1^2+2$ 
which is strictly less than $r^2+1$ when $r \geq 4.$ 
In other words, they form a proper subfamily of the $(r^2+1)$-dimensional family of simple Ulrich bundles. 
Therefore, we see that there is a stable rank $r$ Ulrich bundle on $V_5.$ 
\end{proof}

Using the existence, we can prove that a general member of the above quiver representations gives an Ulrich bundle on $V_5$ .

\begin{lemm}\label{cokernel}
Let $ E_t(-1)=\Coker(U^{\oplus r} \to {Q^*}^{\oplus r}) $ be a generic member of the quiver representation of $\Gamma$. 
Then $E_t$ is an Ulrich bundle of rank $r$ on $V_5.$
\end{lemm}

\begin{proof}
For $r \geq 2$ we know that there is an Ulrich bundle $E_0$ of rank $r$ on $V_5$ from the previous proposition. 
Then 
$E_0$ can be written as $ E_0(-1) \cong \Coker(U^{\oplus r} \to {Q^*}^{\oplus r}) $ by Proposition \ref{E(-1)}.
From the condition $ E_0(-1) \cong \Coker(U^{\oplus r} \to {Q^*}^{\oplus r}) $ and the results in Lemma \ref{computation of cohomology}, 
we have $H^*(V_5,E_0(-1))=H^*(V_5,E_0(-2))=0.$ 
Moreover, we have $H^i(V_5,E_0(-3))=0$ for $i=0,1,2$ and the following exact sequence 
$$ 0 \to H^3(V_5,U(-2)^{\oplus r}) \to H^3(V_5,{Q^*}(-2)^{\oplus r}) \to H^3(V_5,E_0(-3)) \to 0. $$
We already computed that the first and second terms are isomorphic $5r$-dimensional vector spaces since $E_0$ is an Ulrich bundle. 
We know that there is a $3r^2$-dimensional family of morphisms (it is $(r^2+1)$-dimensional if we consider equivalent classes of morphisms) from $U^{\oplus r}$ to ${Q^*}^{\oplus r}.$ 
Because being injective morphism between vector bundles is an open condition, we can see that a generic morphism from $U^{\oplus r}$ to ${Q^*}^{\oplus r}$ is injective. 
Let us consider the cokernel $E_t(-1)$ of a generic morphism from $U^{\oplus r}$ to ${Q^*}^{\oplus r}.$ 
Via the same computation as above we see that $H^*(V_5,E_t(-1))=H^*(V_5,E_t(-2))=0.$ 
Moreover, we have $H^i(V_5,E_t(-3))=0$ for $i=0,1,2$ and get the following exact sequence 
$$ 0 \to H^3(V_5,U(-2)^{\oplus r}) \to H^3(V_5,{Q^*}(-2)^{\oplus r}) \to H^3(V_5,E_t(-3)) \to 0. $$
Because inducing isomorphism $H^3(V_5,U(-2)^{\oplus r}) \to H^3(V_5,{Q^*}(-2)^{\oplus r})$ is an open condition, we see that $H^3(V_5,E_t(-3))=0$ for 
a generic member $E_t(-1)$ of the quiver representation of $\Gamma$. 
Since a sheaf satisfying these vanishing conditions is an Ulrich bundle (cf. \cite{CH}), 
$E_t$ is an Ulrich bundle of rank $r$ on $V_5$.
\end{proof}



Using similar arguments in \cite{CKL}, we have the following result.

\begin{prop}\label{preserving stability}
For any $r \geq 2,$ a (semi-)stable Ulrich bundle corresponds to a (semi-)stable quiver representation. 
\end{prop}

\begin{proof}
Let $E$ be an Ulrich bundle of rank $r$ on $V_5.$ 
Then by Proposition \ref{E(-1)} we have the following sequence
$$ 0 \to U^{\oplus r} \to {Q^*}^{\oplus r} \to E(-1) \to 0 . $$
Let $V_{\bullet}$ be the quiver representation corresponding to $E(-1).$ 
From the definition of Ulrich bundles, we see that $\chi(R_{\bullet},V_{\bullet})=0$ and $H^0(R_{\bullet},V_{\bullet})=0.$ 

Let $V_{\bullet}'$ be a subrepresentation of $V_{\bullet}$ and consider the following exact sequence
$$ 0 \to V_{\bullet}' \to V_{\bullet} \to V_{\bullet}'' \to 0 . $$
Then we have the 
long exact sequence
\begin{align*}
0 & \to H^0(R_{\bullet},V_{\bullet}') \to H^0(R_{\bullet},V_{\bullet}) \to H^0(R_{\bullet},V_{\bullet}'') \\
& \to H^1(R_{\bullet},V_{\bullet}') \to H^1(R_{\bullet},V_{\bullet}) \to H^1(R_{\bullet},V_{\bullet}'') \to 0,
\end{align*}
so this implies that 
$$ H^0(R_{\bullet},V_{\bullet}')=0 \mbox{ and } H^1(R_{\bullet},V_{\bullet}'')=0. $$
Hence $\chi(R_{\bullet},V_{\bullet}') \leq 0$ and $\chi(R_{\bullet},V_{\bullet}'') \geq 0.$ 
Let $\theta \colon K_0(\CC \Gamma\mbox{-mod}) \to \ZZ$ be the character defined by $\theta=\chi(R_{\bullet},-).$ 
From the above inequality, we see that $V_{\bullet}$ is a $\theta$-semistable $\Gamma$-representation.

If $E$ is strictly semistable, then there exists the 
destabilizing sequence
$$ 0 \to E' \to E \to E'' \to 0, $$
where $\mu(E')=\mu(E)$, and from \cite{Beauville, CHGS} we see that $E'$ is also an Ulrich bundle of smaller rank. 
Therefore, 
$V_{\bullet}$ is a strictly semistable $\Gamma$-representation. 

Suppose that $E$ is a stable Ulrich bundle but the corresponding quiver representation $V_{\bullet}$ is strictly semistable. 
Let $W_{\bullet}$ be the proper $\Gamma$-subrepresentation of $V_{\bullet}$ with $\theta(W_{\bullet})=0$ 
and let $U^{\oplus a} \to {Q^*}^{\oplus b}$ be the corresponding element in $\langle U, Q^* \rangle.$ 
Then we have the following commutative diagram.

\begin{displaymath}
\xymatrix{ 
 & U^{\oplus a} \ar[r] \ar[d] & \ar[d] {Q^*}^{\oplus b} \ar[d]  \ar[r] &  F(-1) \ar[d] \\
0 \ar[r] \ar[d] & U^{\oplus r} \ar[r] \ar[d] & {Q^*}^{\oplus r} \ar[r] \ar[d] & \ar[d] E(-1)  \\ 
\mbox{Ker} \ar[r] & U^{\oplus (r-a)} \ar[r] & {Q^*}^{\oplus (r-b)} \ar[r] & \mbox{Cok} }
\end{displaymath}
From the snake lemma, we have the following exact sequence 
$$ 0 \to \mbox{Ker} \to F(-1) \to E(-1) \to \mbox{Cok} \to 0. $$
Because $F$ is a coherent sheaf with $H^i(V_5,F(-j))=0$ for $i=0,1,2,3$ and $j=1,2,3$ from the commutative diagram and $a=b$, 
we see that $F$ is an Ulrich subbundle with $\mu(F) = \mu(E)$ 
which gives a contradiction to the stability of $E$. 
Thus, 
when $E$ is stable, the corresponding $\Gamma$-representation must be also stable.
\end{proof}

From the above discussions, we get the desired result. Let $\Gamma,$ $\theta$ be as above.

\begin{theo} \label{moduli space of Ulrich} 
Let $M(r)$ (resp. $M^s(r)$) be the moduli space of S-equivalence classes of Ulrich bundles (resp. stable Ulrich bundles) of rank $r \geq 2$ on $V_5.$ 
Then the natural functor $\mbox{\rm D}(V_5) \to \mbox{\rm D}(\CC \Gamma\mbox{\rm-mod})$ given by semiorthogonal projection induces 
a map $\varphi \colon M(r) \to \mathcal M_{(r,r)}^{{\theta}\text{-ss}}(\Gamma)$ satisfying the following properties: 
\begin{enumerate}
\item $\varphi$ maps stable (resp. strictly semistable) objects to stable (resp. strictly semistable) objects;  
\item $\varphi$ is an injective map and the image $\varphi(M(r))$ can be described as follows. 
$$
\varphi(M(r)) = \left\{  V_{\bullet} \in \mathcal M_{(r,r)}^{\theta\text{-ss}}(\Gamma) : 
				\begin{array}{ll}
					\Hom(R_{\bullet},V_{\bullet})=0 \text{ and the corresponding} \\
					\text{morphism } U^{\oplus r} \to {Q^*}^{\oplus r} \text{ is injective}
				\end{array}
				\right\}
$$
Moreover, the embedding functor $\mbox{\rm D}(\CC \Gamma\mbox{\rm-mod}) \to \mbox{\rm D}(V_5)$ induces a morphism $\psi \colon \varphi(M(r)) \to M(r)$ which is the inverse of $\varphi.$
\item $\varphi$ induces an isomorphism of $M^s(r)$ onto 
$$
\varphi(M^s(r)) = \left\{  V_{\bullet} \in \mathcal M_{(r,r)}^{\theta\text{-s}}(\Gamma) : 
				\begin{array}{ll}
					\Hom(R_{\bullet},V_{\bullet})=0 \text{ and the corresponding} \\
					\text{morphism } U^{\oplus r} \to {Q^*}^{\oplus r} \text{ is injective}
				\end{array}
				\right\}, 
$$
which forms a smooth $(r^2+1)$-dimensional open subset of $\mathcal M_{(r,r)}^{{\theta}\text{-s}}(\Gamma) \subset \mathcal M_{(r,r)}^{{\theta}\text{-ss}}(\Gamma).$ 
\end{enumerate}
\end{theo}

\begin{proof}
The proof is similar to the proof of \cite[Theorem 3.13]{CKL} and \cite{Qin;V5}. 
We give 
the proof for the convenience of readers. 


From Proposition \ref{E(-1)} and Proposition \ref{preserving stability}, 
we see that there is a map from the set of isomorphism classes of Ulrich bundles on $V_5$ to the set of isomorphism classes of $\theta$-semistable $\Gamma$-representations. 
Therefore, it is enough to prove that this assignment preserves S-equivalence classes in order to show that the map $\varphi$ is well-defined. 
When $E$ is a stable Ulrich bundle, 
we know that the corresponding quiver representation is 
stable by Proposition \ref{preserving stability}. 
Hence $\varphi$ is well-defined for stable Ulrich bundles. 
For a strictly semistable Ulrich bundle $E,$ there is a Jordan--H\"older filtration whose factors are stable Ulrich bundles (cf. \cite{Beauville, CHGS}). 
Then the corresponding quiver representation gives an iterative extension of stable quiver representations which are images of the Jordan--H\"older factors. Therefore, the correspondence preserves S-equivalences and $\varphi$ is well-defined. \\


(1) follows from Proposition \ref{preserving stability}. \\

(2) Suppose that there is an injective morphism $U^{\oplus r} \to {Q^*}^{\oplus r}$ and the cokernel is orhogonal to $\cO(2).$ 
Then the cokernel twisted by $\mathcal O(1)$ is an Ulrich bundle (cf. \cite{LMS1}). 
Therefore, we see that the image $\varphi(M(r))$ 
consists of $V_{\bullet} \in \mathcal M_{(r,r)}^{\theta\text{-ss}}(\Gamma)$ such that $\Hom(R_{\bullet},V_{\bullet})=0$ and the corresponding morphism $U^{\oplus r} \to {Q^*}^{\oplus r}$ is injective, i.e., we have the following identity.
$$
\varphi(M(r)) = \left\{  V_{\bullet} \in \mathcal M_{(r,r)}^{\theta\text{-ss}}(\Gamma) : 
				\begin{array}{ll}
					\Hom(R_{\bullet},V_{\bullet})=0 \text{ and the corresponding} \\
					\text{morphism } U^{\oplus r} \to {Q^*}^{\oplus r} \text{ is injective}
				\end{array}
				\right\}
$$
Then the embedding functor $\mbox{\rm D}(\CC \Gamma\mbox{\rm-mod}) \to \mbox{\rm D}(V_5)$ induces the 
map $\psi \colon \mathcal M_{(r,r)}^{{\theta}\text{-ss}}(\Gamma) \to M(r)$.  
Now let us prove that $\psi$ is a morphism. 
Let us consider the $3r^2$-dimensional space $\cR_{(r,r)}(\Gamma)$ of quiver representations with dimension vector $(r, r)$ and its open subset $\cR^{\theta\text{-ss}}_{(r,r)}(\Gamma)$ of semistable quiver representations. 
Recall that $\mathcal M_{(r,r)}^{{\theta}\text{-ss}}(\Gamma)$ is the good quotient of $\cR^{\theta\text{-ss}}_{(r,r)}(\Gamma).$ 
Next, let us consider the $G_0$-invariant open subset $\cR^{\theta\text{-ss}}_{(r,r)}(\Gamma)^{\circ}$ consisting of $V_{\bullet} \in \mathcal R_{(r,r)}^{\theta\text{-ss}}(\Gamma)$ 
such that $\Hom(R_{\bullet},V_{\bullet})=0$ and the corresponding morphism $U^{\oplus r} \to {Q^*}^{\oplus r}$ is injective. 
Then the cokernels of the morphisms $U^{\oplus r} \to {Q^*}^{\oplus r}$ form a flat family of Ulrich bundles (cf. \cite{CH}) and 
we see that there is a morphism $\widetilde{\psi} \colon \mathcal R_{(r,r)}^{\theta\text{-ss}}(\Gamma)^{\circ} \to M(r).$ 
(Recall that $M(r)$ is a subscheme of the coarse moduli space of semistable sheaves on $V_5$ which corepresents the moduli functor 
sending a scheme $S$ to the set of semistable sheaves on $S \times V_5$ whose fiberwise Hilbert polynomials are the same as those of Ulrich bundles.) 
From a general result 
of Geometric Invariant Theory (cf. \cite{Newstead}), 
we see that $\varphi(M(r))$ is a good quotient (hence a categorical quotient) of $\mathcal R_{(r,r)}^{\theta\text{-ss}}(\Gamma)^{\circ}$ and 
hence there is a morphism $\psi \colon \varphi(M(r)) \to M(r).$ 
We can check that $\psi$ is the inverse of $\varphi$ set-theoretically. \\

(3) From Proposition \ref{preserving stability}, we see that the image $\varphi(M^s(r))$ 
consists of $V_{\bullet} \in \mathcal M_{(r,r)}^{\theta\text{-s}}(\Gamma)$ such that $\Hom(R_{\bullet},V_{\bullet})=0$ and the corresponding morphism $U^{\oplus r} \to {Q^*}^{\oplus r}$ is injective. 
For $E \in M^s(r),$ there is the following natural isomorphism 
$$ T_{E}M^s(r) \cong \Ext^1(E,E) \cong \Ext^1(\varphi(E),\varphi(E)) \cong T_{\varphi(E)} \mathcal M_{(r,r)}^{\theta\text{-s}}(\Gamma) $$
which is induced by the projection functor. 
Note that both $M^s(r)$ and $\varphi(M^s(r))$ are 
$(r^2+1)$-dimensional smooth quasi-projective varieties. 
Since we are working on the complex number field, we see that $\psi \colon \varphi(M^s(r)) \to M(r)$ is an open embedding whose image is $M^s(r)$ by 
Zariski's main theorem. 
Especially,  $\psi \colon \varphi(M^s(r)) \to M^s(r)$ is an isomorphism. 
Therefore, we obtain the desired result.
\end{proof}

It is remarkable that the moduli spaces of rank $r(\geq 2)$ Ulrich bundles on cubic threefolds, intersections of two 4-dimensional quadrics, and $V_5$ are all smooth quasi-projective varieties of dimension $r^2+1$ 
(see \cite{CKL, LMS1} for more details). 
We conjecture that for any smooth Fano 3-fold of index 2 the moduli space of rank $r (\geq 2)$ stable Ulrich bundles is nonempty and is a smooth quasi-projective variety of dimension $r^2+1.$ 
By Lemma 2.3 of \cite{CFM}, 
we deduce that if there is a stable Ulrich bundle of rank $r$ on any smooth Fano 3-fold $X$ of index 2, 
then the corresponding moduli space of stable Ulrich sheaves on $X$ at that point is smooth of dimension $r^2+1$. 


Note that smooth Fano 3-folds of Picard number 1 and index 2 were classified (Theorem 3.3.1 of \cite{IP}).  
Moreover, in many cases the ideas used in this paper are similar to those in \cite{CKL, LMS1}, 
especially the relation between stability of Ulrich bundles on Fano 3-folds and a different kind of stability of objects in the semiorthogonal components. 
Recently, Bayer, Lahoz, Macr\`{i} and Stellari introduced a method to induce Bridgeland stability conditions on semiorthogonal components in \cite{BLMS}. 
We think there might be a uniform proof using their method for the results in this paper and \cite{CKL, LMS1}, which might work for more general Fano varieties. \\

It is also an interesting task to study modular compactifications of the moduli spaces of Ulrich bundles discussed in this paper. 
Druel studied moduli spaces of semistable sheaves of rank 2 with $c_1=0, c_2=2$ and $c_3=0$ on cubic 3-folds in \cite{Druel} and it gives a natural compactification of the moduli space of Ulrich bundles of rank 2 on the cubic 3-fold. Recently, Qin studied moduli spaces of instanton sheaves on Fano 3-folds $V_4$ and $V_5$ in \cite{Qin;V5, Qin;V4}, and showed that there are similar descriptions for moduli spaces for these cases. 
Via this he obtained compactifications of the moduli spaces of stable rank 2 Ulrich bundles discussed in \cite{CKL} and this paper. 
It is an interesting question whether similar techniques will give us compactifications of moduli spaces of higher rank Ulrich bundles on these and other Fano 3-folds.

\bigskip

\end{document}